\documentclass[12pt]{article}
\usepackage[cp1251]{inputenc}
\usepackage[french, english]{babel}
\usepackage{amsfonts}
\usepackage{amsthm}
\usepackage{amsbsy}
\usepackage{amssymb}
\usepackage{graphicx}
\usepackage{eso-pic}
\usepackage{dsfont}
\usepackage{tikz}
\usepackage{xfrac}
\usepackage{float}
\usepackage[bookmarks=true]{hyperref}
\usepackage{bookmark}
\usepackage{amsmath}
\usepackage{subfigure}
\usepackage{hyperref}
\usepackage{scalerel,stackengine}
\usetikzlibrary{matrix,arrows}

\usepackage[letterpaper,left=0.85in,right=0.85in,top=1in,bottom=1in]{geometry}
\newtheorem{lemma}{Lemma}[section]
\newtheorem{teo}[lemma]{Theorem}
\newtheorem{prop}[lemma]{Proposition}
\newtheorem{cor}[lemma]{Corollary}

\theoremstyle{definition}
\newtheorem{defn}[lemma]{Definition}

\theoremstyle{remark}
\newtheorem{rem}[lemma]{Remark}

\newcommand{\matR} {\ensuremath {\mathbb{R}}}
\newcommand{\matQ} {\ensuremath {\mathbb{Q}}}
\newcommand{\matZ} {\ensuremath {\mathbb{Z}}}
\newcommand{\matC} {\ensuremath {\mathbb{C}}}

\newcommand{\matH} {\ensuremath {\mathbb{H}}}

\newcommand{\detr}{\ensuremath {\rm det}}
\newcommand{\SO}{\ensuremath {\rm SO}}

\newcommand{\SL}{\ensuremath {\rm SL}}
\newcommand{\GL}{\ensuremath {\rm GL}}
\newcommand{\Or}{\ensuremath {\rm O}}
\newcommand{\POr}{\ensuremath {\rm PO}}
\newcommand{\Isom}{\ensuremath {\rm Isom}}

\newcommand{\modu}{{\rm mod}}

\newcommand{\PSL}{\rm PSL}

\stackMath
\newcommand\reallywidehat[1]{%
\savestack{\tmpbox}{\stretchto{%
  \scaleto{%
    \scalerel*[\widthof{\ensuremath{#1}}]{\kern-.6pt\bigwedge\kern-.6pt}%
    {\rule[-\textheight/2]{1ex}{\textheight}}
  }{\textheight}%
}{0.5ex}}%
\stackon[1pt]{#1}{\tmpbox}%
}
\begin{document}
\title{Embedding arithmetic hyperbolic manifolds}
\author{Alexander Kolpakov, Alan W. Reid, Leone Slavich}
\date{}

\maketitle

\begin{abstract}
\noindent We prove that any arithmetic hyperbolic $n$-manifold 
of simplest type can either be \emph{geodesically embedded} into an 
arithmetic hyperbolic $(n+1)$-manifold or its universal $\modu~2$ Abelian
cover can. \end{abstract}

\section{Introduction}
Throughout the paper, all manifolds are assumed to be connected unless otherwise specifically stated. A complete, orientable, finite volume hyperbolic $n$-manifold $M$  {\em bounds geometrically} if $M$ is realized as the totally geodesic boundary of a complete, orientable, finite volume, hyperbolic $(n+1)$-manifold $W$. There has been recent interest in a number of problems related to geometric boundaries (c.f. \cite{KMT}, \cite{KS}, \cite{LR}, \cite{Sl0} and \cite{Sl}). 

One could ask for a weaker property, namely, that $M$ simply embeds totally geodesically
in a complete, orientable, finite volume, hyperbolic $(n+1)$-manifold $W$ (which was considered in \cite{Mar}). In this case cutting $W$ along $M$ either produces a single copy of $M$ that bounds geometrically (if $M$ happens to be separating in $W$), or two copies of $M$ that bound some hyperbolic $(n+1)$-manifold
(otherwise).

In order to state the main result of this note, let us introduce the following notation. For a finitely generated group $\Gamma$ we let $\Gamma^{(2)} = \langle \gamma^2 : \gamma\in \Gamma \rangle$. 
We also refer the reader to \S \ref{section:prelim} for the other necessary definitions of terms in Theorem \ref{main}.

\begin{teo}
\label{main}
Let $M={\matH}^n/\Gamma$ ($n\geq 2$) be an orientable arithmetic hyperbolic n-manifold of simplest type. 
\begin{enumerate}
\item If $n$ is even, then $M$ embeds as a totally geodesic submanifold of an orientable arithmetic hyperbolic $(n+1)$-manifold $W$. 
\item If $n$ is odd, the manifold $M^{(2)}={\matH}^n/\Gamma^{(2)}$ embeds as a totally geodesic submanifold of an orientable arithmetic hyperbolic $(n+1)$-manifold~$W$. 
\end{enumerate}
Moreover, when $M$ is not defined over $\mathbb{Q}$ (and is therefore closed), the manifold $W$ can be taken to be closed.
\end{teo}

The reason for the even--odd distinction will become apparent below (see \S \ref{section:comments-adjoint}).  The last sentence in Theorem \ref{main} reflects, for example, the fact that certain closed arithmetic hyperbolic $3$-manifolds of simplest type arising from anisotropic quadratic forms defined over $\matQ$ can only embed in non-compact arithmetic hyperbolic $4$-manifolds, since by Meyer's Theorem (see \cite[\S 3.2, Corollary 2]{Se}) all integral quadratic forms of signature $(4,1)$ are isotropic.

It is known that when $n$ is even all arithmetic subgroups of 
$\Isom({\matH}^n)$ are of simplest type (see \cite{VS}) and so we deduce the following.

\begin{cor}
\label{alleven}
Assume that $n$ is even and let $M={\matH}^n/\Gamma$ be an orientable arithmetic hyperbolic n-manifold. 
Then $M$ embeds as a totally geodesic submanifold of an orientable arithmetic hyperbolic $(n+1)$-manifold 
$W$. 
\end{cor}

An interesting question is to understand if the property of embedding geodesically is preserved under commensurabilty. While not providing definitive evidence for such a hypothesis, we give the first examples of commensurability classes of hyperbolic manifolds all of whose members are geodesically embedded.

It is also known that every non-compact, arithmetic, hyperbolic $n$-\linebreak manifold is of simplest type \cite{LiM}, so another corollary of Theorem \ref{main} is

\begin{cor}
\label{noncompact}
Let $M={\matH}^n/\Gamma$ ($n\geq 2$) be an orientable non-compact arithmetic hyperbolic $n$-manifold.
\begin{enumerate}
\item If $n$ is even, then $M$ embeds as a totally geodesic submanifold of an orientable arithmetic hyperbolic $(n+1)$-manifold $W$. 
\item If $n$ is odd, the manifold $M^{(2)}={\matH}^n/\Gamma^{(2)}$ embeds as a totally geodesic submanifold of an orientable arithmetic hyperbolic $(n+1)$-manifold~$W$. 
\end{enumerate}
\end{cor}

We can also address the question of bounding. A more precise version of the theorem below (including odd dimensions) is given in \S \ref{section:proof-main2}.

\begin{teo}
\label{main2}
Suppose that $n\geq 2$ is even and let $M={\matH}^n/\Gamma$ be an orientable arithmetic hyperbolic n-manifold of simplest type which double covers a non-orientable hyperbolic manifold. Then $M$ bounds geometrically.\end{teo}

This is interesting even in dimension $2$ where we provide new examples of arithmetic hyperbolic 2-manifolds that bound geometrically (see \S \ref{section:surfaces}).

Also, in dimension $3$, in \cite{KS2} and \cite{Sl0} certain arithmetic hyperbolic link complements are shown to bound geometrically, and in \cite{Sl} it is shown that the figure-eight knot complement bounds geometrically. Our techniques provide other examples.

\begin{cor}
\label{noncompact2}
Let $M={\matH}^3/\Gamma$ be a hyperbolic 3-manifold which is a finite cover of a Bianchi orbifold 
${\matH}^3/\PSL(2,O_d)$. Then $M$ embeds as a totally geodesic submanifold of an orientable arithmetic hyperbolic $4$-manifold.
\end{cor}

In many cases we are also able to prove that $M$ bounds geometrically. Unlike the constructions of \cite{KS}, \cite{Sl0} and \cite{Sl}, our methods are algebraic and similar to \cite{LR}. Indeed, Theorems \ref{main} and \ref{main2} extend the results in \cite{LR} which provide for all $n$ examples of orientable arithmetic hyperbolic $n$-manifolds of simplest type that embed (in fact, bound geometrically) in an orientable arithmetic hyperbolic $(n+1)$-manifold. The crucial new ingredient that we rely on is the recent work of Agol and Wise on proving separability, as applied in \cite{BHW} in the setting of arithmetic groups of simplest type.
\section*{Acknowledgements}
{The first and third authors were partially supported by the Italian FIRB project ``Geometry and topology of low-dimensional manifolds'', RBFR10\-GHHH. The first author wishes to thank William Jagy for his introduction into the vast body of literature on quadratic forms and related topics. The second author was supported in part by an NSF grant and The Wolfensohn Fund administered through the Institute for Advanced Study. He would also like to thank the I.A.S and M.S.R.I. for their hospitality whilst working on this project. He also wishes to thank Yves Benoist (Universit\'{e} Paris-Sud) and Gopal Prasad (University of Michigan at Ann Arbor) for helpful conversations and correspondence. The third author was supported by a grant from "Scuola di scienze di base Galileo Galilei", and wishes to thank the Department of Mathematics, Universit\`a di Pisa, for their hospitality whilst working on this project. He also wishes to thank Bruno Martelli and Stefano Riolo (Universit\`a di Pisa) for several helpful conversations. All three authors thank the I.C.T.P., Trieste, Italy, and the organizers of  the ``Advanced School on Geometric Group Theory and Low-Dimensional Topology: Recent Connections and Advances'' held at the I.C.T.P. in June 2016, which allowed them to work on this project at an early stage.}

\section{Preliminaries on algebraic and arithmetic groups}\label{section:prelim}

We recall some basic terminology about linear algebraic groups. We refer the reader to \cite{PR} for further details.

Throughout this section $G\subset \GL(n,\matC)$ will be a semi-simple linear algebraic group defined over a totally real number field $k$ (with ring of integers $R_k$). Associated with $G$ is its adjoint algebraic group 
$\overline{G}$, which is also defined over $k$, and the homomorphism $\pi: G\rightarrow \overline{G}$ is a central $k$-isogeny (see \cite[Section 2.6.1]{Ti}).

For $U$ a subring of $\matC$ we let $G(U)=G\cap \GL(n,U)$ denote the subgroup of $U$-points of $G$. By an {\em arithmetic subgroup of $G$} we mean a subgroup of $G(\matR)$ commensurable with $G(R_k)$. The proposition below is a useful starting point for what is to follow.

\begin{prop}
\label{subgroupsofkpoints}
Let $G$ and $H$ be semi-simple linear algebraic groups defined over $k$ with $G\subset H$. Let $\Gamma < G(k)$ be an arithmetic subgroup. Then there exists an arithmetic subgroup $\Lambda < H(k)$ with 
$\Gamma < \Lambda$.
\end{prop}

\begin{proof} Since $\Gamma$ is arithmetic it is commensurable with $G(R_k)$ and so $\Gamma \cap G(R_k)$ is a subgroup of finite index in $\Gamma$. By definition $H(R_k)$ is an arithmetic subgroup of $H$, and evidently $\Gamma \supset \Gamma \cap H(R_k) \supset \Gamma \cap G(R_k)$, and so $\Gamma \cap H(R_k)$ is an arithmetic subgroup of $\Gamma$.

Set $\Lambda_1 = \bigcap_{\gamma\in \Gamma}\gamma H(R_k)\gamma^{-1}$. Since $\Gamma<G(k) < H(k)$, each element $\gamma\in\Gamma$ commensurates $H(R_k)$, and so by the previous paragraph $\Lambda_1$ is a finite intersection of $H(k)$-conjugates of $H(R_k)$, which are all therefore commensurable arithmetic subgroups of $H$. Hence $\Lambda_1$ is an arithmetic subgroup of $H$, and moreover $\Lambda_1$ is normalized by
$\Gamma$. It follows that $\Lambda = \Gamma \cdot \Lambda_1$ is the required arithmetic subgroup of $H$ contained in $H(k)$. \end{proof}

There are many situations where all arithmetic lattices of $G$ are contained in $G(k)$. For example, it is a result of Borel \cite{Bo} (see also \cite{Allan}) that this takes place if $G$ is a centreless linear algebraic group, and so this holds in particular if $G=\overline{G}$.

However, this is often not the case, which we illustrate in the following example.\\[\baselineskip]
\noindent{\bf Example.}~Let $\Gamma = \SL(2,\matZ)$, and let $n>1$ be a square-free integer. We
can construct arithmetic groups $\Gamma_n < \SL(2, \mathbb{R})$ commensurable with $\SL(2,\matZ)$ that are not subgroups of $\SL(2,\matQ)$.  

Let $\Gamma_0(n)$ be the subgroup of $\SL(2,\matZ)$ consisting of matrices which are congruent to an upper triangular matrix modulo $n$. Then the element 
$$\tau_n = \begin{pmatrix}0 & 1/\sqrt{n}\\ & \\ -\sqrt{n} & 0 \end{pmatrix}$$
normalizes $\Gamma_0(n)$ and so the group $\Gamma_n = \langle \Gamma_0(n), \tau_n \rangle$
is an arithmetic subgroup of $\SL(2,\matR)$ commensurable with $\Gamma$ and obviously not contained in 
$\SL(2,\matQ)$.

Using this example, we can easily show that a more general statement of Proposition \ref{subgroupsofkpoints} for embedding arithmetic groups in other arithmetic groups does not hold.\\[\baselineskip]
\noindent{\bf Example:} Embed $\SL(2,\matR) < \SL(3,\matR)$ as shown below:
$$\left( \begin{array}{ccc}   \SL(2,\matR) &| & 0 \\      
             
 0     & |  &   1 \end{array} \right),$$ 

\noindent and consider the arithmetic subgroup $\Gamma_n < \SL(2,\matR)$ 
constructed above with $\Gamma_n \hookrightarrow \SL(3,\matR)$ embedded as
$$\left( \begin{array}{ccc}   \Gamma_n &| & 0 \\      
             
 0     & |  &   1 \end{array} \right).$$ 

Since $\SL(3,\matC)$ is centreless, it follows from the aforementioned result of Borel that this subgroup of $\SL(3,\matR)$ cannot be a subgroup of any arithmetic subgroup of $\SL(3,\matR)$ commensurable with $\SL(3,\matZ)$.  However, it can also be seen directly as follows. 

With $\tau_n$ the involution as above, set
$$\gamma_n = \left( \begin{array}{ccc}   \tau_n &| & 0 \\      
             
 0     & |  &   1 \end{array} \right).$$

If $\gamma_n \in \Lambda$, an arithmetic group commensurable with $\SL(3,\matZ)$, then $\Lambda$ contains a normal subgroup $N$ which is a finite index subgroup of $\SL(3,\matZ)$. Hence $N$ contains an element of the form:

$$x=\begin{pmatrix}1 & 0 & P\\ 0 & 1 & 0 \\ 0 & 0 & 1\end{pmatrix},$$
for some $P\in\matZ$. Since $N$ is a normal subgroup of $\Lambda$, we have that $\gamma_n x\gamma_n^{-1} \in N$. However,

$$\gamma_n x\gamma_n^{-1} = 
\begin{pmatrix} 1 & 0 & 0\\ 0 & 1 & -P\sqrt{n} \\ 0 & 0 & 1 \end{pmatrix},$$
which is never an element of $\SL(3,\matZ)$.

\section{Arithmetic subgroups of orthogonal groups}

In this section we adapt some of the discussion of \S \ref{section:prelim} to (special) orthogonal groups of quadratic forms.

\subsection{Quadratic forms and arithmetic lattices}\label{section:quad-forms-arithm-lat}
Let $k$ be a totally real number field of degree $d$ over $\matQ$ equipped with a fixed embedding into
$\matR$ which we refer to as the identity embedding, and denote the ring of integers of $k$ by $R_k$.  Let $V$ be an $(n+1)$-dimensional vector space over $k$ equipped with a non-degenerate quadratic form $\mathrm{f}$ defined over $k$ which has signature $(n,1)$ at the identity embedding, and signature $(n+1,0)$ at the remaining $d-1$ embeddings. 

Given this, the quadratic form $\mathrm{f}$ is equivalent over $\matR$ to the quadratic form $x_0^2+x_1^2+\cdots +x_{n-1}^2-x_n^2$, and for any non-identity Galois embedding $\sigma:k\rightarrow \matR$, the quadratic form $\mathrm{f}^\sigma$ (obtained by applying $\sigma$ to each entry of $\mathrm{f}$) is equivalent over $\matR$  to $x_0^2+x_1^2+\cdots +x_{n-1}^2+x_n^2$. We call such a quadratic form {\em admissible}.

Let $F$ be the symmetric matrix associated with the quadratic form $\mathrm{f}$ and let $\Or(f)$ (resp. $\SO(f)$) denote the linear algebraic groups defined over $k$ described as:
\begin{align*}
\Or(f)&=\{X\in\GL(n+1,\matC):X^tFX=F\}\quad\hbox{and}\\
\SO(f)&=\{X\in\SL(n+1,\matC):X^tFX=F\}.
\end{align*}

For a subring $L\subset \matC$, we denote the $L$-points of $\Or(f)$ (resp. $\SO(f)$) by $\Or(f,L)$ (resp. $\SO(f,L)$). An {\em arithmetic lattice} in $\Or(f)$ (resp. $\SO(f)$) is a subgroup $\Gamma<\Or(f)$ commensurable with $\Or(f,R_k)$ (resp. $\SO(f,R_k)$). Note that an arithmetic subgroup of $\SO(f)$ is an arithmetic subgroup of $\Or(f)$, and an arithmetic subgroup $\Gamma< \Or(f)$ determines an arithmetic subgroup $\Gamma\cap \SO(f)$ in $\SO(f)$.

Apart from the case of $n = 3$, $\SO(f)$ is a connected absolutely almost simple semi-simple algebraic group defined over $k$. It is centreless when $n$ is even, but has a non-trivial centre when $n$ is odd.

\section{Arithmetic groups of simplest type} 

Next we describe a particular construction of arithmetic subgroups of the Lie group $\Isom({\matH}^n)$ via admissible quadratic forms.

\subsection{Some notation}\label{section:notation}

Let $J_n$ denote both the quadratic form $x_0^2+x_1^2+\cdots +x_{n-1}^2-x_n^2$, and the diagonal matrix associated with the form. We identify hyperbolic space ${\matH}^n$ with the upper half-sheet $\{x\in {\matR}^{n+1} : J_n = -1, x_n>0\}$ of the hyperboloid $\{x\in {\matR}^{n+1} : J_n = -1\}$, and letting
$$\Or(n,1)=\{X\in\GL(n+1,\matR) : X^tJ_nX=J_n\},$$

\noindent
we can identify $\Isom({\matH}^n)$ with the subgroup of $\Or(n,1)$ preserving the upper half-sheet of the hyperboloid $\{x\in {\matR}^{n+1} : J_n = -1\}$, denoted by
$\Or^{+}(n,1)$.

Moreover $\Or^{+}(n,1)$ is isomorphic to $\POr(n,1)$ (the central quotient of $\Or(n,1)$). With this notation, $\Isom^+({\matH}^n) = \SO^+(n,1)$, the index $2$ subgroup in $\Or^{+}(n,1)$, which is the connected component of the identity of $\Or(n,1)$.

Note that when $n$ is even, $\Isom({\matH}^n)$ can be identified with $\SO(n,1)$ by observing that there exists an isomorphism $\varphi : \Or^{+}(n,1)\rightarrow \SO(n,1)$ defined by $\varphi(T)=(\detr T)\, T$, which restricts to the identity on $\SO^+(n,1)$.

\subsection{Some arithmetic groups}\label{section:arithm-gps}

To pass to arithmetic subgroups of $\Or^{+}(n,1)$ and $\SO^{+}(n,1)$, we first note from \S \ref{section:quad-forms-arithm-lat} that, given an admissible quadratic form defined over $k$ of signature $(n,1)$, there exists $T\in \GL(n+1,\matR)$ such that $T^{-1}\Or(f,\matR)T = \Or(n,1)$. 

A subgroup $\Gamma < \Or^+(n,1)$ is called {\em arithmetic of simplest type} if $\Gamma$ is commensurable with the image in $\Or^+(n,1)$ of an arithmetic subgroup of $\Or(f)$ under the conjugation map above. An arithmetic hyperbolic $n$-manifold $M={\matH}^n/\Gamma$ is called {\em arithmetic of simplest type} if $\Gamma$ is. The same set-up using special orthogonal groups constructs orientation-preserving arithmetic groups of simplest type (and orientable arithmetic hyperbolic $n$-manifolds of simplest type).

We will also use the following notation. For a subring $R\subset \matR$, let $\Or^+(f,R)$ (resp. $\SO^+(f,R)$) be the subgroup of $\Or(f,R)$ (resp. $\SO(f,R)$) that leave both components of the cone $\{x\in {\matR}^{n+1} : \mathrm{f}(x) < 0\}$ invariant.  Note that since $T$ provides an equivalence of the form $\mathrm{f}$ with $J_n$, it follows that $\mathrm{f}(x)=J_n(Tx)$, and $T^{-1}\Or^+(\mathrm{f},\matR)T = \Or^+(n,1)$.

Notice that when $n$ is even, arithmetic subgroups of simplest type in $\Or^+(n,1)$ can be identified with the image of an arithmetic subgroup of $\SO(f)$ (using the discussion in \S \ref{section:notation}). Moreover, since $\SO(f)$ is centreless, as
remarked above, by \cite{Bo}, all of its arithmetic subgroup are contained in $\SO(f,k)$. Therefore, when $n$ is even any arithmetic group of simplest type in $\Or^+(n,1)$ is constructed from an arithmetic subgroup of $\SO(f,k)$ for some admissible quadratic form $\mathrm{f}$.


\subsection{Locating arithmetic subgroups}\label{section:locating-arithm-gps}

In this subsection we prove the following result that will be an important step towards proving Theorem \ref{main}.

\begin{prop}
\label{orthogonalgroupsembedded}
Let $\mathrm{f}$ be an admissible quadratic form of signature $(n,1)$ defined over a totally real field $k$, $\mathrm{g}$ an admissible quadratic form of signature $(n+1,1)$ defined over the same field $k$ and suppose that $\mathcal{G} = \SO(f) < \mathcal{H} = \SO(g)$. Let $\Gamma$ be an arithmetic subgroup of $\mathcal{G}$. Then:
\begin{enumerate}
\item If $n$ is even, $\Gamma$ embeds in an
arithmetic subgroup of $\mathcal{H}$.
\item If $n$ is odd, $\Gamma^{(2)}$ embeds in an
arithmetic subgroup of $\mathcal{H}$.\end{enumerate}
\end{prop}

\begin{proof} Part (1) follows from Proposition \ref{subgroupsofkpoints} and the previous remark that for $n$ even, all arithmetic subgroups of $\mathcal{G}$ are contained in $\mathcal{G}(k)$. The second part follows from the proof of \cite[Lemma 10]{ERT} which shows that $\Gamma^{(2)}$ is a subgroup of $\mathcal{G}(k)$, and then we argue as in the even-dimensional case.\end{proof}

Indeed, both statements of Proposition \ref{orthogonalgroupsembedded} also apply to arithmetic subgroups of 
$\Or^{+}(f,\mathbb{R})$.  This can be deduced from \cite[Lemmas 6 and 10]{ERT}.

We record the following immediate corollary of Proposition~\ref{orthogonalgroupsembedded} and the remark above. 

\begin{cor}
\label{simplestembed}
Let $\Gamma$ be an arithmetic subgroup of $\Or^{+}(n,1)$ of simplest type arising from an admissible quadratic form $\mathrm{f}$ of signature $(n,1)$ defined over a totally real field $k$. Suppose that there is an admissible quadratic form $\mathrm{g}$ of signature $(n+1,1)$ defined over the same field $k$, with $\Or(f)< \Or(g)$. Then:
\begin{enumerate}
\item If $n$ is even, $\Gamma$ embeds in an arithmetic subgroup of $\Or^{+}(n+1,1)$ of simplest type.
\item If $n$ is odd, $\Gamma^{(2)}$ embeds in an arithmetic subgroup of $\Or^{+}(n+1,1)$ of simplest type.\end{enumerate}
\end{cor}

The proof of Theorem \ref{main} will follow from Corollary \ref{simplestembed} once we have arranged the set-up of algebraic groups required by its statement, together with certain
separability results. The former is done in \S \ref{section:embedding}, and the latter in \S \ref{section:separability} and \S \ref{section:main-proof}.

Before that, we finish this section with some comments about adjoint groups in this context.

\subsection{Comments on adjoint groups}\label{section:comments-adjoint}

We begin by discussing one of the issues on embedding arithmetic groups when $n$ is odd. 

Thus, as before, let $\mathrm{f}$ be an admissible quadratic form of signature $(n,1)$ defined over a totally real field $k$ and $\mathrm{g}$ an admissible quadratic form of signature $(n+1,1)$ defined over $k$, as well. As above, let $\mathcal{G} = \SO(f) < \SO(g) = \mathcal{H}$. If $n$ is even, then the inclusion map 
$\iota: \mathcal{G} \rightarrow \mathcal{H}$ descends to the adjoint groups $\overline{\mathcal{G}} \rightarrow \overline{\mathcal{H}}$. The reason for this is that $\overline{\mathcal{G}}=\mathcal{G}$ and
so although the centre of $\mathcal{H}$ is non-trivial it is disjoint from $\mathcal{G}$. However, when $n$ is odd the inclusion map does not extend, precisely because in this case $\overline{\mathcal{H}}=\mathcal{H}$
and the central quotient is actually an isomorphism. In particular the centre of $\mathcal{G}$ is preserved; i.e. the restriction of the adjoint map $\mathcal{H}\rightarrow \overline{\mathcal{H}}$ to 
$\mathcal{G}$ is not the adjoint map.

Regardless of $n$, the subgroup $\SO^+(f,\matR) < \SO(f,\matR)$ described in \S \ref{section:arithm-gps} is centreless, and so this does inject under the adjoint map. However, it is still not clear that one can arrange the image of an arithmetic group to lie in the $k$-points.\\[\baselineskip]
We now provide a number of examples to highlight some of the issues. The first example serves as a counterpoint to the discussion about $\SL(2,\matZ)$ given in~\S \ref{section:prelim}.\\[\baselineskip]
\noindent{\bf Example.}~Let $\mathrm{f}(x,y,z)=xz-y^2$. Clearly $\mathrm{f}$ is an isotropic quadratic form of signature $(2,1)$. Indeed, in the notation above, $\SO^+(f,\matZ)$ is isomorphic to $\PSL(2,\matZ)$ and it can be seen from \cite[Section 2]{Mac} that every maximal arithmetic subgroup commensurable with 
$\PSL(2,\matZ)$ is mapped isomorphically (by a homomorphism that is defined over $\matQ$) to a subgroup of 
$\SO^+(f,\matQ)$. That is to say, these maximal groups can be realised as subgroups of the $\matQ$-points
unlike the situation of the embedding of 
$\SL(2,\matZ) \hookrightarrow \SL(3,\matZ)$ described in \S \ref{section:prelim}.\\[\baselineskip]
We now consider a similar example when $n=3$.\\[\baselineskip]
\noindent{\bf Example.}~Let $\pi = \langle 2+i \rangle \subset O_1={\matZ}[i]$ be one of the
prime ideals in $O_1$ of norm $5$. Following the set-up in \S \ref{section:prelim}, let
$\Gamma_0(\pi)$ be the subgroup of $\SL(2,O_1)$ consisting of matrices
which are congruent to an upper triangular matrix modulo $\pi$.  Then
the element
$$\tau = \begin{pmatrix}0 & 1/\sqrt{2+i}\\ & \\ -\sqrt{2+i} & 0 \end{pmatrix}$$
normalizes $\Gamma_0(\pi)$ and so the group $\Gamma = \langle \Gamma_0(\pi),\tau \rangle$
is an arithmetic subgroup of $\SL(2,\matC)$ commensurable with $\SL(2,O_1)$ and not contained in $\SL(2,{\matQ}(i))$. For convenience, we will continue to use the same notation on passage to $\PSL(2,O_1)$.

Let $\mathrm{f}$ denote the quadratic form $x_0x_1+x_2^2+x_3^2$. The homomorphism
$\phi:\PSL(2,\matC) \rightarrow \SO^+(f,\matR)$ given below:
\begin{align*}
&\begin{pmatrix} 
a_0+a_1i&b_0+b_1i\\ 
c_0+c_1i&d_0+d_1i 
\end{pmatrix} \longmapsto\\
&\resizebox{\textwidth}{!}{$\begin{pmatrix} d_0^2+d_1^2 & -b_0^2-b_1^2 & b_0d_0+b_1d_1 & b_1d_0-b_0d_1 \\
-c_0^2-c_1^2 & a_0^2+a_1^2 & -a_0c_0-a_1c_1 & -a_1c_0+a_0c_1 \\
2(c_0d_0+c_1d_1) & -2(a_0b_0+a_1b_1) & 2(b_0c_0+a_1d_1)+1 & 2(b_1c_0-a_0d_1) \\
2(c_0d_1-c_1d_0) & 2(a_1b_0-a_0b_1) & 2(b_1c_0-a_1d_0) & 2(a_0d_0-b_0c_0)-1 
\end{pmatrix}$}
\end{align*}
describes an isomorphism in which $\PSL(2,O_1)$ maps into
$\SO^+(f,\matZ)$.  However notice that $\tau$ has image with
$(1,2)$-entry equal to $- \left|\frac{1}{\sqrt{2+i}}\right|^2 = \frac{-1}{\sqrt{5}}$, and so the image of $\Gamma$ is not contained in $\SO^+(f,\matQ)$ either.  

Note that, as is easily checked, $\Gamma^{(2)} < \SO^+(f,\matQ)$.

\begin{rem}\hspace{-1ex} The isomorphism described above was worked out by Michelle Chu whilst correcting a claimed isomorphism in \cite[p. 463]{EGM} which turned out to be incorrect.\end{rem}

\section{Embedding orthogonal groups}\label{section:embedding}

In this section we describe how to arrange the embedding of orthogonal groups required in \S \ref{section:locating-arithm-gps}.

\begin{prop}
\label{extend_forms}
Let $\mathrm{f}$ be an admissible quadratic form of signature $(n,1)$ defined over a totally real field $k$.  Then there is an admissible quadratic form $\mathrm{g}$ of signature $(n+1,1)$ over $k$, so that $\Or(f)<\Or(g)$ (and $\SO(f)<\SO(g)$).
\end{prop}

\begin{proof} To prove Proposition \ref{extend_forms}, we first reduce it to the case when $\mathrm{f}$ is a 
diagonal quadratic form, namely:\\[\baselineskip]
\noindent{\bf Claim:}~{\em Suppose that $\mathrm{f}$ is represented by the admissible diagonal quadratic form $a_0x_0^2+a_1x_1^2+\cdots +a_{n-1}x_{n-1}^2-bx_n^2,$ where $a_i\in R_k$ are all positive and square free for $i=0,\ldots, n-1$, and $b\in R_k$ is positive and square free. Then there is an admissible 
diagonal quadratic form $\mathrm{g}$ of signature $(n+1,1)$ with $\Or(f)<\Or(g)$.}\\[\baselineskip]
We defer the proof of the Claim and deduce Proposition \ref{extend_forms}. To that end, let $\mathrm{f}$ be an admissible quadratic form, so that standard properties of quadratic forms (see \cite{La}) provide an equivalence over $k$ of $\mathrm{f}$ to an admissible diagonal quadratic form $\mathrm{f}_0$. In particular,
there exists $T\in\GL(n+1,k)$ so that $T^{-1}\mathrm{O}(\mathrm{f}_0)T=\Or(f)$. Assuming the Claim, and applying it to 
$\mathrm{f}_0$, we have an admissible diagonal quadratic form $\mathrm{g}_0$ so that $\Or(f_0)<\Or(g_0)$. We can extend $T$ in a natural way to define a matrix 
$$\widehat{T} = \left( \begin{array}{ccc}   1 &| & 0 \\      
             
 0     & |  &   T \end{array} \right) \in  \GL(n+2,k),$$

\noindent which provides an equivalence of the diagonal form $\mathrm{g}_0$ to an admissible 
quadratic form $\mathrm{g}$ with $\Or(f)<\Or(g)$.\\[\baselineskip]
We now prove the Claim. Let us start with some comments about the form $\mathrm{f}$. If $\mathrm{f}$ is anisotropic over $k$ then we can assume that $b\neq a_i$ for $i=0,\ldots , n-1$.  Since $\Or(\lambda f)=\Or(f)$ for all $\lambda \in k^*$, we can multiply $\mathrm{f}$ by ${a_0}^{-1}$ and therefore assume that 
$a_0=1$, and also that all other coefficients are square-free. 

Suppose first that $k\neq \matQ$. In this case we can take $\mathrm{g}=y^2+\mathrm{f}$, which will be again a quadratic form over $k$. Then $\Or(g,R_k)$ 
is cocompact, as follows from \cite[Proposition 6.4.4]{Morris}. 

Now suppose $k=\matQ$. By Meyer's Theorem (see \cite[\S 3.2, Corollary 2]{Se}), if $n\geq 4$ (i.e. $\mathrm{f}$ is an indefinite quadratic form over $\matQ$ in at least $5$ variables) then $\mathrm{f}$ is isotropic and so taking $\mathrm{g} = y^2 + \mathrm{f}$ is an admissible isotropic quadratic form.  When $n=3$, since by Meyer's theorem any admissible quadratic form of signature $(4,1)$ will be isotropic, we can simply take $\mathrm{g} = y^2 + \mathrm{f}$ once again. In either case, $\Or(g,\matZ)$ has finite co-volume \cite[Theorem 7.8]{BHCh} but is not cocompact \cite[Proposition 5.3.4]{Morris}.

When $n=2$, then we set $\mathrm{g} = q\,y^2 + \mathrm{f}$, where $\mathrm{f}$ is a ternary quadratic form of signature $(2,1)$ and $q$ is a positive rational number. If $\mathrm{f}$ is isotropic, then so is $\mathrm{g}$, and 
we can simply put $q=1$. As above, in this case, the group $\Or(g,\matZ)$ 
has finite co-volume but is not cocompact. If $\mathrm{f}$ is anisotropic, then there exists a positive rational number $q$ such that $-q$ is not represented by $\mathrm{f}$ over $\mathbb{Q}$ (see Lemma~\ref{lemma:missing-number} 
of the Appendix). The form $\mathrm{g} = q\,y^2 + \mathrm{f}$ is the required anisotropic quadratic form. The group $\Or(g,\matZ)$ is cocompact by \cite[Proposition 5.3.4]{Morris}. \end{proof}  

We can improve the embedding described in Proposition \ref{extend_forms} in the following sense.

\begin{cor}
\label{improved_embedding}
In the notation established above, if $R\subset \matR$ is a subring, then
$\Or^{+}(\mathrm{f},R)$ can be embedded in $\SO^+(\mathrm{g},R)$.\end{cor}

\begin{proof} As above we write $\mathrm{g}=ay^2+\mathrm{f}$ for some $a\in R_k$. Define
a homomorphism $\psi : \Or^{+}(\mathrm{f},R) \rightarrow \SO(\mathrm{g},R)$ by 
$$\widehat{M}= \psi(M) = \left( \begin{array}{ccc}   \det(M) &| & 0 \\      
             
 0     & |  &   M \end{array} \right).$$

We claim that $\widehat{M}\in\SO^+(\mathrm{g},R)$. It is clear that $\mathrm{g}(\widehat{M}(v))=\mathrm{g}(v)$ for every $v=(y,x_0,\dots,x_n)\in \mathbb{R}^{n+2}$, since $M\in \Or(\mathrm{f})$, and also $\mathrm{det}(\widehat{M})=1$. Moreover, since the transformation $M\in \Or^{+}(\mathrm{f},R)$ preserves the upper half-sheet of the $n$-dimensional hyperboloid $\{\mathrm{f}(x_0,\dots,x_n)$ $=-1\}\cap\{ x_n>0\}$, its image $\widehat{M}$ preserves the upper half-sheet of the $(n+1)$-dimensional hyperboloid $\{\mathrm{g}(y,x_0,\dots,x_n)=-1\}\cap \{x_n>0\}$. \end{proof}

\noindent{\bf Example:}~Let $\mathrm{g}^\prime = x^2 + y^2 + z^2 - 7
w^2$. This is an anisotropic form over $\mathbb{Q}$ due to the fact
that $7$ is not a sum of three integer squares (see for example
\cite[Appendix, p.~45]{Se}), and a number is a sum of three integers
if and only if it's a sum of three rational squares by the
Davenport-Cassels lemma (see \cite[Lemma B, p.~46]{Se}).  Thus,
$\SO^+(\mathrm{g}^\prime, \mathbb{Z})$ defines a cocompact arithmetic
Kleinian group $H$ of simplest type.

Now we can embed a totally geodesic sub-2-orbifold into the quotient space $\mathbb{H}^3 / H$ as follows. First, consider the ternary form $\mathrm{f} = x^2 + z^2 - 3w^2$. This form is also anisotropic over $\mathbb{Q}$ by reasons analogous to the above. Thus, it defines a cocompact arithmetic Fuchsian group $G$ of simplest type. Set $\mathrm{g} = 21 y^2 + \mathrm{f}$. 
We claim that the form $\mathrm{g}$ is equivalent to $\mathrm{g}^\prime$ over $\mathbb{Q}$. 

To see this, since the signatures of $\mathrm{g}$ and
$\mathrm{g}^\prime$ coincide, it suffices to check their discriminants and Hasse invariants  \cite[Theorem~9,
p.~44]{Se}. We have that $d(\mathrm{g}) = -7\cdot 3^2$ and
$d(\mathrm{g}^\prime) = -7$. Also, we compute $\epsilon(\mathrm{g}) =
(21, -3) = (-(-3)\cdot 7,-3) = (7, -3) = 1$, by the standard Hilbert
symbol's properties \cite[Proposition 2, p.~19]{Se}. Analogously,
$\epsilon(\mathrm{g}^{\prime}) = (1, -7) = 1$, by a straightforward
check, over any $\mathbb{Q}_p$. 

Thus, we have an example of an arithmetic Kleinian group containing an arithmetic Fuchsian subgroup, whose embedding can be observed on the level of the corresponding quadratic forms as described at the end of the
of the proof of Proposition \ref{extend_forms}. 

\section{Separability and consequences}\label{section:separability}

To pass from embedding of subgroups to actual embeddings of manifolds
we need to use recent progress on certain separability properties of these
arithmetic groups of simplest type. 

\subsection{Separability}

\begin{defn}
  Let $\Gamma$ be a finitely generated, discrete subgroup of
  $\Or^+(n,1)$. We say that $\Gamma$ is
\emph{geometrically finite extended residually finite} (or GFERF for
  short) if every geometrically finite subgroup $H$ of $\Gamma$ is
  separable in $\Gamma$ (i.e. $H$ is closed in the profinite topology
  on $\Gamma$).\end{defn}

We record the following from \cite{BHW}.

\begin{teo}\label{teo:standard-gferf}
Let $\Gamma<\Or^+(n,1)$ be an arithmetic group of simplest type. 
Then $\Gamma$ is GFERF.
\end{teo}

The proof of Theorem \ref{teo:standard-gferf} follows from
\cite[Corollary $1.12$]{BHW} (see also the final remark of the paper,
which deals with the non-compact case). 

An important consequence of GFERF that we will make use of is the 
following. Throughout we will denote the profinite completion of a group $G$ 
by $\widehat{G}$, and the closure of a subset $X\subset\widehat{G}$ by $Cl(X)$.

\begin{lemma}\label{lemma:closure-isomorphic}
Let $\Gamma<\Or^+(n,1)$ be an arithmetic group of simplest type, and 
let $H<\Gamma$ be a geometrically 
finite subgroup. Then $Cl(H)$ in $\widehat{\Gamma}$ is
isomorphic to $\widehat{H}$.\end{lemma}

\begin{proof} By the
universal property of profinite completions, there is a natural
continuous surjective homomorphism $\phi:\widehat{H}\rightarrow Cl(H)$
which extends the inclusion $H<\Gamma$ (notice that
$Cl(H)$ is a closed subgroup of a profinite group and therefore is
also profinite). 

We need to show that the map $\phi$ is injective. 
Note that $\phi$ is injective if and only if, given any 
finite index subgroup $H_1<H$, there is a finite index subgroup 
$\Gamma_1<\Gamma$ such that $\Gamma_1\cap H < H_1$. 

By standard properties, all of the finite index subgroups of $H$
are also geometrically finite and therefore separable in $\Gamma$. This implies that indeed, for any finite index subgroup $H_1$ in $H$, there exists a finite index subgroup $\Gamma_1<\Gamma$ such that $\Gamma_1\cap H=H_1$. For details, see \cite[Lemma~4.6]{Reid} and the preceding discussion. \end{proof}


\subsection{Cohomological goodness}

We now introduce a particular class of groups which have the nice property that their profinite completion is torsion-free.

Let $G$ be a profinite group, and $M$ a discrete $G$-module (i.e.\ an
Abelian group equipped with the discrete topology on which $G$ acts
continuously). Let $C^n(G,M)$ be the set of all continuous maps from
$G^n$ to $M$. Define the coboundary operator $\partial:
C^n(G,M)\rightarrow C^{n+1}(G,M)$ in the usual way to produce a
cochain complex $C^*(G,M)$. The cohomology groups of this complex are
denoted by $H^q(G,M)$ and are called the continuous cohomology groups
of $G$ with coefficients in $M$.

\begin{defn}
Let $\Gamma$ be a finitely generated group. We say that $\Gamma$ is \emph{good} if, for all $q\geq 0$ and for every finite $\Gamma$-module $M$ the natural homomorphism of cohomology groups \begin{equation}H^q(\widehat{\Gamma},M)\rightarrow H^q(\Gamma,M)\end{equation} induced by the inclusion $\Gamma\rightarrow \widehat{\Gamma}$
is an isomorphism between the cohomology of $\Gamma$ and the continuous cohomology of $\widehat{\Gamma}$.
\end{defn}

In general, establishing whether a group is good or not, is
difficult. However, in our context we have the following (which is
proved below).

\begin{teo}
\label{Simplest_type_good}
Let $\Gamma < \Or^+(n,1)$ be an arithmetic group of simplest type. Then $\Gamma$
is good.\end{teo}

Indeed this follows for all finite volume hyperbolic 3-manifolds
by the virtual fibering results of Agol \cite{Ag} and Wise \cite{Wi},
together with the
following criterion due to Serre \cite[Chapter 1, Section 2.6]{Serre}:

\begin{prop}\label{prop:criterion-good}
Let $\Gamma$ be a finitely generated group. Suppose that there exists a short exact sequence
\begin{equation}
1\rightarrow N\rightarrow \Gamma\rightarrow H \rightarrow 1
\end{equation} where $H$ and $N$ are good groups, $N$ is finitely generated and the cohomology group $H^q(N,M)$ is finite for every $q$ and every finite $\Gamma$-module $M$. Then $\Gamma$ is good.
\end{prop}
Moreover, goodness is preserved by commensurability and free products of good groups are good. Furthermore, controlled amalgams of good groups are good 
\cite{GJJZ}.
In particular (see e.g. \cite[Theorem 7.3]{Reid}), it can be shown that 
the following groups are good: 
\begin{enumerate}
\item finitely generated Fuchsian groups,
\item fundamental groups of compact $3$-manifolds,
\item fully residually free groups,
\item Right Angled Artin groups (RAAG's),
\end{enumerate}

We will make use of the previous remarks and 
the following lemma (\cite[Lemma 3.1]{MZ}) in the proof of Theorem
\ref{Simplest_type_good}.

\begin{lemma}
\label{retract_good}
Suppose that $G$ is a residually finite good group and $H$ is a virtual retract
of $G$. Then $H$ is good.\end{lemma}

\begin{proof}[Proof of Theorem \ref{Simplest_type_good}]
The proof of Theorem \ref{teo:standard-gferf} in 
\cite{BHW} actually shows that if $\Gamma<\SO^+(n,1)$
is an arithmetic lattice of simplest type,
then there exists a finite
index subgroup $\Gamma_1<\Gamma$ so that $\Gamma_1$ is a subgroup of a right-angled Coxeter group $C$ (see also \cite[Theorem 1.10]{BHW}). 
Moreover, the nature
of the construction in \cite{BHW} (see also \cite{CDW} and \cite{Haglund})
shows that $\Gamma_1$ sits as a quasi-convex subgroup of $C$. Since by \cite{Haglund} Coxeter groups virtually retract onto their quasi-convex subgroups, there exists a finite index subgroup $C_1<C$ together
with a retraction homorphism $r:C_1\rightarrow \Gamma_1$. 

The result will now follow from Lemma \ref{retract_good} once we establish that
$C$ is good (since as remarked goodness is preserved by commensurability). 
This is a consequence of the fact that any Coxeter group $C$ is virtually special
in the sense of \cite{HW} (see in particular 
\cite[Proposition 3.10]{HW}, and \cite{HW2}). More precisely, \cite{HW} provides a finite index subgroup
$C_2<C$ with $C_2$ a quasi-convex subgroup of a RAAG $A$, and hence
again by \cite{Haglund} (see also \cite{HW})
a finite index subgroup $A_2$ admitting a retraction onto
$C_2$. The group $A$ is good since it is a RAAG and so Lemma \ref{retract_good}
applies to show that $C_2$ is good. Since goodness is preserved by commensurability, $C$ is good as required. \end{proof}

The implication of goodness that we require is the
following \cite[Corollary 7.6]{Reid}:

\begin{prop}\label{prop:good+rf->torsionfree}
Suppose that $\Gamma$ is a residually finite, good group of finite cohomological dimension over $\mathbb{Z}$. Then $\widehat{\Gamma}$ is torsion-free.
\end{prop}

The fundamental groups of hyperbolic manifolds are well-known to be
residually finite by Malcev's theorem, and any hyperbolic manifold is
a $K(\pi_1(M),1)$ space.  Hence, this implies that $\pi_1(M)$ is of
finite cohomological dimension over $\mathbb{Z}$.  Using the remarks
above on good groups, and Proposition~\ref{prop:good+rf->torsionfree}
we therefore have the following corollary:

\begin{cor}\label{cor:closure-torsion-free}
Let $M$ be a hyperbolic manifold with good fundamental group. Then 
$\widehat{\pi_1(M)}$ is torsion-free.
\end{cor}

\section{Completing the proof of Theorem \ref{main}}\label{section:main-proof}

Given Proposition \ref{extend_forms} and Corollary \ref{simplestembed}, 
the following result will complete the 
proof of Theorem \ref{main}.
The proof of Proposition \ref{torsionfree_contain} is given in a much more
general context in \cite{DBMR}, however, for completeness we
give a proof in \S \ref{section:torsionfree_contain_proof} adapted to the case at hand.  

\begin{prop}
\label{torsionfree_contain}
Let $\Gamma$ be a torsion-free arithmetic lattice in $\Or^+(n,1)$, of 
simplest type, and $\Lambda$ an arithmetic lattice of simplest type
in $\Or^+(n+1,1)$ such that $\Gamma<\Lambda$. Then there is a torsion-free
subgroup $\Lambda_1<\Lambda$ of finite index such that $\Gamma<\Lambda_1$. 
\end{prop}

Deferring the proof of Proposition \ref{torsionfree_contain} for now,
we complete the proof of Theorem \ref{main}.

\begin{proof}[Proof of Theorem \ref{main}] 
Let $M={\matH}^n/\Gamma$
be an orientable arithmetic hyperbolic $n$-manifold of simplest type. 
We will assume that $n$ is even. The odd-dimensional case is handled exactly the same
way replacing $\Gamma$ by $\Gamma^{(2)}$.

By Corollary \ref{simplestembed}, there exists exists an arithmetic
lattice $\Lambda$ of simplest type in $\SO^+(n+1,1)$ such that
$\Gamma<\Lambda$.
By Proposition \ref{torsionfree_contain}, we can find
a torsion-free subgroup $\Lambda_1<\Lambda$ with $\Gamma<\Lambda_1$.
Now by Theorem \ref{teo:standard-gferf}, $\Lambda_1$ is GFERF, and a standard
consequence of this (see \cite{Scott1}) is that $M$ embeds in
a finite sheeted cover of ${\matH}^{n+1}/\Lambda_1$.  

The final sentence in the statement of Theorem \ref{main} follows from the proof of the Claim in Proposition \ref{extend_forms} and the fact that closed arithmetic manifolds of simplest type are associated with quadratic forms either over a finite extension $k$ of $\mathbb{Q}$, $k \neq \mathbb{Q}$, or with quadratic forms over $\mathbb{Q}$ which are anisotropic. 

As we have seen above, isotropic quadratic forms over $\mathbb{Q}$ give rise to necessarily non-compact arithmetic manifolds. This completes 
the proof.\end{proof}

\subsection{Proof of Proposition \ref{torsionfree_contain}}\label{section:torsionfree_contain_proof}

In what follows $\Gamma$ will always denote an arithmetic lattice of 
simplest type in $\Or^+(n,1)$.

\begin{lemma}\label{lemma:conjugacytorsion}
Suppose that $\eta \in \Gamma$ has finite order, and $[\eta]$ denotes the
conjugacy class of $\eta$ in $\Gamma$. 
Then $\text{Cl}([\eta])\subset \widehat{\Gamma}$
consists entirely of elements of finite order.
\end{lemma}

\begin{proof}
Let $\lambda$ be an element of $\text{Cl}([\eta])$. Then there exists a convergent sequence $\{\eta_j \}_{j=0}^{\infty}$ of elements of $[\eta]$ whose limit is $\lambda$. For each $\eta_j$, by definition there exists an element 
$\beta_j\in \Gamma$ such that $\beta_j^{-1} \eta_j \beta_j=\eta$.

The sequence $\{\beta_j\}_{j=0}^{\infty}$ is a sequence in the compact group 
$\widehat{\Gamma}$ and therefore admits a  convergent subsequence $\{\beta_l\}_{l=0}^{\infty}$ with limit $\beta$. By continuity of taking inverses, we see that also $\{\beta_l^{-1}\}_{l=0}^{\infty}$ converges to $\beta^{-1}$. Finally:
\begin{equation}
\eta=\text{lim}\,\beta^{-1}_l\cdot \text{lim}\, \eta_l \cdot \text{lim}\,\beta_{l}=\beta^{-1}\lambda \beta. 
\end{equation}
Therefore $\eta$ and $\lambda$ are conjugate in $\widehat{\Gamma}$, 
and since $\eta$ has finite order so does $\lambda$.
\end{proof}

We now commence with the proof of Proposition \ref{torsionfree_contain}.

\begin{proof}
Let $\Gamma<\Lambda$ be as in the statement of Proposition 
\ref{torsionfree_contain}, and $h\in \Lambda$ 
be a non-trivial element of finite order. By
Lemma \ref{lemma:conjugacytorsion}, the closure of the conjugacy class
of $h$ in $\widehat{\Lambda}$ consists entirely of
elements of finite order. By Lemma~\ref{lemma:closure-isomorphic} and
Corollary \ref{cor:closure-torsion-free}, the closure of $\Gamma$ in
$\widehat{\Lambda}$ is torsion-free. Therefore:
\begin{equation}\label{eq:nointersection}
Cl(\Gamma)\cap Cl([h])=\emptyset.
\end{equation}

From this we can conclude that there is a finite-index subgroup $H$ of 
$\Lambda$ which contains $\Gamma$ and is disjoint from $[h]$. To see this,
suppose to the contrary, and fix a total ordering $\{H_1,H_2,\dots\}$ on the finite-index subgroups of $\Lambda$ which contain $\Gamma$. We have that $\Gamma = \bigcap^{\infty}_{i=1}\, H_i$, since $\Gamma$ is separable in $\Lambda$. Then, for every index $m$, we can find an element $h_m$ in $[h]$ which belongs to the intersection of all $H_i$ with $i\leq m$.

The closure $Cl([h])$ is compact and therefore, up to passing to a convergent subsequence, we can suppose $\lim_{m\to \infty}h_m=h\in Cl([h])$. The limit $h$ necessarily belongs to the closures of all the $H_i$'s, and therefore to $Cl(\Gamma)$, 
which is a contradiction because of (\ref{eq:nointersection}).

Now since $\Lambda$ is a lattice, 
there exist only a finite number $[h_1],\dots,[h_n]$ of conjugacy classes of 
non-trivial elements of finite order in $\Lambda$. To see this, 
first recall that being a lattice, $\Lambda$ admits a convex
finite sided fundamental polyhedron $D$. 
Now, up to conjugation, we can suppose that 
a non-trivial element of finite order $h$ fixes a face $F$ of $D$. 
Such a face $F$ will be adjacent to only a
finite number of other copies of $D$, say $D_1,\dots,D_k$, and the
element $h$ necessarily maps $D$ to one of the $D_i$'s, say $D_l$. Moreover, no other element of $\Lambda$ maps $D$ to
$D_l$.  Since there are only a finite number of choices for the face
$F$ and for the domain $D_l$, the result follows.

For each such class $[h_i]$, there exists a finite index subgroup 
$F_i<\Lambda$ which contains $\Gamma$ and is disjoint from $[h_i]$.
Clearly the intersection 
\begin{equation}
\Lambda_1=\bigcap_{i=1}^n F_i
\end{equation} 
is a finite index subgroup of $\Lambda$ which contains $\Gamma$  
and is torsion-free. This completes the proof.
\end{proof}

\section{Proof of Theorem \ref{main2}}\label{section:proof-main2}

Throughout this section $\mathrm{f}$ is an admissible quadratic form of signature $(n,1)$
and $N={\matH}^n/H$ a non-orientable manifold double-covered by 
$M={\matH}^n/\Gamma$ with $H$ (and hence $\Gamma$)
arithmetic of simplest type.

Theorem \ref{main2} will be a consequence of the following theorem on
recalling from \S \ref{section:arithm-gps} that when $n$ is even all arithmetic groups of
simplest type arise from $\SO(f,k)$.

\begin{teo}
\label{generalmain2}
In the notation above, if $H$ is conjugate into $\Or^+(f,k)$ then
$N$ bounds geometrically.\end{teo}

\begin{proof} For convenience we identify $H<\Or^+(f,k)$, and so 
$\Gamma < \SO^+(f,k)$. 
Using Corollary \ref{simplestembed} we can find an arithmetic hyperbolic
$(n+1)$-orbifold $X={\matH}^{n+1}/\Lambda$ with $H<\Lambda$. As in the proof Theorem \ref{main} we can use 
Proposition \ref{torsionfree_contain} to 
pass to a further finite-sheeted cover $X_1\rightarrow X$
so that $N$ embeds. Indeed, using Corollary \ref{improved_embedding}
$X_1$ can be chosen to be orientable, and 
the proof is now completed by
the well-known fact:
\begin{lemma}[\cite{LR}]
\label{untwisting_nonorientable}
Suppose that $N$ is a codimension $1$ non-orientable embedded
totally geodesic submanifold of an orientable hyperbolic manifold $X$, and $M$ is the orientation cover of $N$. Then $M$ bounds geometrically.
\end{lemma}\vspace{-2em}\end{proof}

\section{Low dimensions}

In this section we discuss bounding in the case of $n=3$ and $n=2$.

\subsection{Bianchi groups} The statement of Corollary \ref{noncompact2} will follow once we identify the groups $\PSL(2,O_d)$ as subgroups of $\SO(f,\matZ)$ for certain isotropic quaternary quad\-ratic forms $\mathrm{f}$, c.f. \cite[Theorem 2.1]{JM}. Thus, we avoid passing to a $\mathrm{mod}\, 2$-cover. In order to arrange for bounding we take $M$ to be a double cover of a non-orientable manifold $N={\matH}^3/H$ with $H$ as in Theorem \ref{generalmain2}.

\subsection{Surfaces}\label{section:surfaces}
The question as to which closed surfaces bound
geometrically a hyperbolic 3-manifold has been of some interest.  It
follows from the work of Brooks \cite{Br} that the set of surfaces
that bound geometrically is a countable dense subset of the appropriate
moduli space, but beyond a few examples (c.f. \cite{Zim1} and
\cite{Zim2}) little by way of explicit families of
surfaces are known to bound geometrically. Indeed, \cite{Zim2} is
concerned with finding such explicit families where the surfaces embed
totally geodesically. One consequence of Theorems \ref{main} and
\ref{main2} is the following.

\begin{cor}
\label{surfacebound}
Let $\Sigma={\matH}^2/\Gamma$ be a closed arithmetic hyperbolic 2-manifold. Then $\Sigma$ embeds as a totally geodesic surface in an arithmetic hyperbolic 3-manifold. If $\Sigma$ double covers a non-orientable surface then $\Sigma$ bounds geometrically.
\end{cor}

Since the $(2,3,7)$ triangle group is arithmetic, any Hurwitz surface is arithmetic. Hence as a special case of Corollary \ref{surfacebound} we recover the following result of B. Zimmermann \cite{Zim2}.

\begin{cor}
\label{hurwitzbound}
Let $\Sigma$ be a Hurwitz surface. Then $\Sigma$ embeds as a totally geodesic surface in an arithmetic hyperbolic 3-manifold.
\end{cor}

Corollary \ref{hurwitzbound} applies in particular to the Klein quartic curve. However, by \cite{Sing} the Klein quartic does not double cover a non-orientable surface and so we cannot deduce from Theorem~\ref{generalmain2} that the Klein quartic bounds.  However, we can use the results of \cite{Sing}
to deduce the following.

\begin{cor}
\label{newhurwitzbound}
There are infinitely many Hurwitz surfaces that bound geometrically.
\end{cor}

\section{Appendix}

In this appendix, we prove the following statement:

\begin{lemma}\label{lemma:missing-number}
Let $\mathrm{f}$ be an indefinite ternary anisotropic quadratic form defined over $\mathbb{Q}$. Then there exist infinitely many rational numbers of either sign not represented by $\mathrm{f}$ over $\mathbb{Q}$.
\end{lemma}

\begin{proof}

The Hasse-Minkowski theorem \cite[Chapter 6.3]{La} implies that $\mathrm{f}$ does not represent (over $\mathbb{Q}$) a rational number $q$  if and only if $\mathrm{f}$ does not represent $q$ over a $p$-adic completion $\mathbb{Q}_p$ of $\mathbb{Q}$ (which is $\mathbb{R}$ if $p = \infty$). 

We first deal with the case $q=0$. By \cite[Theorem 6, p.~36]{Se}, a ternary form $\mathrm{f}$ does not represent $0$ over $\mathbb{Q}_p$ if and only if 
\begin{equation}\label{eq:hilbertsymbol}
(-1,-d(\mathrm{f}))\neq\epsilon(\mathrm{f}) \mbox{ (in } \mathbb{Q}^{*}_p/(\mathbb{Q}^{*}_p)^2\mbox{),}
\end{equation} 
where $d(\mathrm{f})$ is the discriminant of $\mathrm{f}$,
$(\circ, \circ)$ denotes the Hilbert symbol, and
$\epsilon(\mathrm{f})$ is the Hasse invariant of the form
$\mathrm{f}$.

Now suppose $q\neq 0$. By the Corollary after \cite[Theorem 6, p.36]{Se} a ternary form $\mathrm{f}$ does \emph{not} represent $q$ over the $\mathbb{Q}_p$ if and only if the following two conditions are satisfied:
\begin{gather}\label{eq:condition1} 
q = -d(\mathrm{f}) \mbox{ (as elements of } \mathbb{Q}^{\ast}_p/(\mathbb{Q}^{\ast}_p)^2\mbox{),} \\
\label{eq:condition2} 
(-1,-d(\mathrm{f}))\neq \epsilon(\mathrm{f}).
\end{gather} 

Since $\mathrm{f}$ is anisotropic, the latter condition \eqref{eq:condition2} (which does not depend on the choice of $q$) is satisfied in some $\mathbb{Q}_p$ because of \eqref{eq:hilbertsymbol}. Also, since $\mathrm{f}$ is indefinite, we have that $p \neq \infty$. 

It remains to show that, given a prime number $p$, we can always find both a positive and a negative rational number $q$ such that the former condition \eqref{eq:condition1} is satisfied.

First, let us show that one can realise both positive and negative integers as squares in $\mathbb{Q}_p$. Let $n$ be any integer which is co-prime with $p$. Such a number $n$ is a $p$-adic unit in $\mathbb{Z}_p$, \emph{i.e.} an invertible element of the ring of $p$-adic integers.

Suppose $p\neq 2$ first. By \cite[Theorem 3, p.~17 ]{Se}, $n$ as above is a square in $\mathbb{Q}_p$ if and only if its image in $U/U_1=F^{\ast}_p$ is a square. Here, $U = \mathbb{Z}^*_p$ is the group of $p$-adic units, $U_1 = 1 + p\, \mathbb{Z}_p$ and $F_p = \mathbb{Z}/p \mathbb{Z}$. However, the latter is simply the first term in the $p$-adic expansion of $n$. It is therefore sufficient to choose $n$ such that it is a quadratic residue in $F^{\ast}_p$, and there are clearly both positive and negative choices for such an integer $n$.

If $p=2$, by \cite[Theorem 4, p.~18]{Se} in order for a $2$-adic unit $n$ to be a square it is necessary and sufficient that $n = 1\, (\text{mod}\; 8)$ or, equivalently, that the third term in the $2$-adic expansion of $r$ is $1$. Once more, we can clearly pick $n$ to be either positive or negative.

This proves that for every prime number $p$ there are infinitely many integers of either sign, which are realised as squares in $\mathbb{Q}_p$.

Finally, we choose a rational number $q$ (positive or negative) such that $-q/d$ is a square in $\mathbb{Q}_p$ so that \eqref{eq:condition1} is satisfied. Therefore $\mathrm{f}$ does not represent $q$ over $\mathbb{Q}_p$ and, by Hasse-Minkowski, over $\mathbb{Q}$. \end{proof}

\begin{rem}
The argument above actually shows that if the coefficients of the form $\mathrm{f}$ are integers, so that also the determinant $d$ is, then it is possible to find both positive and negative \textit{integers} not represented by the form $\mathrm{f}$.
\end{rem}

\vspace*{0.25in}

\begin{table}[!hb]
\centering
\begin{tabular}{lll}
\begin{tabular}[c]{@{}l@{}}\it Alexander Kolpakov\\ \it Institut de Math\'{e}matiques\\ \it Universit\'{e} de Neuch\^{a}tel\\ \it Rue Emile-Argand 11\\ \it CH-2000 Neuch\^{a}tel\\ \it Suisse/Switzerland\\ \it kolpakov.alexander(at)gmail.com\end{tabular} & \begin{tabular}[c]{@{}l@{}} \it Alan Reid\\ \it Department of Mathematics\\ \it Rice University\\
\it 6100 Main Street\\ \it Houston, TX 77005\\ \it USA\\ \it alan.reid(at)rice.edu\end{tabular} & \begin{tabular}[c]{@{}l@{}} \it Leone Slavich\\ \it Department of Mathematics\\ \it University of Pisa\\ \it Largo Pontecorvo 5\\ \it I-56127 Pisa\\ \it Italy\\ \it leone.slavich(at)gmail.com\end{tabular}
\end{tabular}
\end{table}

\end{document}